\documentclass{amsart}
\usepackage{amsmath}
  \usepackage{graphics} 
  \usepackage{epsfig} 
\usepackage{graphicx}  \usepackage{epstopdf}
 \usepackage[colorlinks=true]{hyperref}

\usepackage{stackrel}
\usepackage[utf8]{inputenc}
\usepackage{amssymb}
\usepackage{amsthm}
\usepackage{tikz}
\usepackage{mathrsfs}
\usepackage{epsf}
\usepackage{bookmark,hyperref}
\hypersetup{
     colorlinks=true,
     linkcolor=blue,
     filecolor=blue,
     citecolor = blue,
     urlcolor=cyan,
     }
\usepackage{pstricks}

\newtheorem{theorem}{Theorem}[section]
\newtheorem{corollary}{Corollary}
\newtheorem*{main}{Main Theorem}
\newtheorem{lemma}[theorem]{Lemma}
\newtheorem{proposition}{Proposition}

\theoremstyle{definition}

\newcommand{\Fi}{\mathrm{ F}}
\newcommand{\gi}{\mathrm{ g}}

\newcommand{\R}{{\mathbb R}}

\title[Lyapunov instability in Lagrangian dynamics] 
      {On the Lyapunov instability in Lagrangian dynamics}

\author[J. M. Burgos  and  M. Paternain]{}



 \email{burgos@math.cinvestav.mx}
 \email{miguel@cmat.edu.uy}

\thanks{The first author has a CONACYT research fellowship. The second author is supported by the FCE-ANII-135352 grant.}

\thanks{$^*$ Corresponding author: xxxx}

\begin{document}
\maketitle

\centerline{\scshape Juan M. Burgos $^*$}
\medskip
{\footnotesize

 \centerline{Departamento de Matem\'aticas, CINVESTAV-\,CONACYT.}
   \centerline{Av. Instituto Polit\'ecnico Nacional 2508, Col. San Pedro Zacatenco, 07360,}
   \centerline{Ciudad de M\'exico, M\'exico.}
} 

\medskip

\centerline{\scshape Miguel Paternain}
\medskip
{\footnotesize
 \centerline{Centro de Matem\'atica, Facultad de Ciencias, Universidad de la Rep\'ublica.}
   \centerline{Igu\'a 4225, 11400, Montevideo, Uruguay.}
}

\bigskip

 \centerline{(Communicated by the associate editor name)}

\begin{abstract}
In the context of mechanical Lagrangian dynamics, we prove a new Lyapunov instability criterion for a non strict local minimum equilibrium point of a smooth potential where the sufficient condition for instability is the existence of a smooth solution of a certain linear PDE derived from the mechanical Lagrangian governing the dynamics. In the presence of a magnetostatic field, we also give an additional sufficient condition for the motion of a charged particle to be Lyapunov unstable.
\end{abstract}

\section{Introduction}

In this paper we give a new criterion for the instability of an equilibrium point of the Lagrangian dynamics of

\begin{equation}\label{Lagrangian}
L(x,v)= Q_x(v)+\mu_x(v)-U(x),\qquad (x,v)\in TN.
\end{equation}
where $Q$ is a positive definite $C^{2}$ quadratic form, $\mu$ is a $C^{2}$ differential one form and the potential $U$ is a nonnegative real valued $C^{2}$ function such that $M=U^{-1}(0)$ is a nonempty closed set.

Every point in $M\times\{\bf 0\}$ is an equilibrium point and Routh's Theorem \cite{Routh} implies that every isolated point of this set is Lyapunov stable for being a strict local minimum of the potential. This theorem states that a strict minimum of a potential $U$ is a Lyapunov stable equilibrium point for the Lagrangian dynamics of the Lagrangian \eqref{Lagrangian}. For mechanical Lagrangians, the theorem was stated by Lagrange in \cite{Lagrange} and proved by Dirichlet in \cite{Dirichlet}.

Even if the equilibrium point in $M\times \{\bf 0\}$ were not isolated, it could be Lyapunov stable. The first example of this phenomenon was given by Painlev\'e in 1904 \cite{Pai} in Newtonian dynamics for a single degree of freedom and a more striking example by Laloy in \cite{Laloy1}, again in the context of Newtonian dynamics but with two degrees of freedom.

The Lagrange-Dirichlet converse for real analytic potentials is a long standing open conjecture posed by Lyapunov in \cite{Ly}. A related open conjecture posed in the form of a problem by Arnold in \cite{Arnold}, problem 1971-4, concerns the instability of an isolated non minimum critical point of the potential in the context of Newtonian dynamics.

Starting from Lyapunov \cite{Ly} and continuing with \cite{Ref1}, \cite{Ha}, \cite{Ref2}, \cite{Ref3}, \cite{Ref4}, \cite{Ref5}, \cite{Ref6}, \cite{Ref7}, many partial results have been given towards the solution of the Arnold's conjecture and their common feature is that the Lyapunov instability criteria involves the lack of a local minimum at the origin of the first nonzero jet of the potential. However, these instability criteria are not sufficient neither to prove the general case nor to prove the case of a non strict local minimum of the potential.

Concerning the case of a non strict local minimum of the potential, in the recent paper \cite{BMP} it was proved the instability on every point of a hypersurface where the potential reaches a minimum in the context of Newtonian dynamics. In this paper, besides of being in the context of Lagrangian dynamics, none manifold structure is assumed in the set $M$ and in particular it could have singularities.


Consider the mechanical Lagrangian:
\begin{equation}\label{Mech_Lagrangian}
L_{mech}(x,v)= Q_x(v)-U(x),\qquad (x,v)\in TN
\end{equation}
where $Q$ is a positive definite $C^{2}$ quadratic form and $U$ is a nonnegative real valued $C^{2}$ function. Define the Riemannian metric $\rho$ on $N$ such that $Q_x(v)=\Vert v\Vert_x^{2}/2$ and denote by $\nabla_\rho$ the gradient with respect to this metric.

In this paper, we prove the following theorem

\begin{theorem}\label{main1}
Consider a zero potential point $p$ and suppose there is a neighborhood $V$ of $p$ in $N$ and a real valued $C^{3}$ regular function $f$ on $V$ such that
$$dU(\,\nabla_\rho f\,)=O(U).$$
Then, $(p,{\bf 0})$ is a Lyapunov unstable equilibrium point of the Lagrangian dynamics of \eqref{Mech_Lagrangian}.
\end{theorem}

In Theorem \ref{main1} the instability is obtained in those cases in which there exists a smooth
solution $f$ of a linear partial differential equation of the
form
\begin{equation}\label{PDE}
\sum_{a,b=1}\,\gi^{ab}\,\partial_a U\,\partial_b f = O(U).
\end{equation}

Corollaries \ref{Generalization} and \ref{Generalization2} provide a family of potentials $U$ for which such solutions exist, even in the neighborhood of points belonging to a subset of the singular set of $U^{-1}(0)$. Thus we can show instability in cases where previous methods do not.

Here, $(\gi^{ab})$ is the inverse matrix of $(\gi_{ab})$ whose terms are the coefficients of the metric $\rho$ expressed with respect to some coordinate neighborhood of the point $p$ mentioned in the theorem.

Theorem \ref{main1} generalizes the result in \cite{BMP} for Newtonian dynamics. Specifically, in contrast to \cite{BMP} where all of the critical points are regular, Theorem \ref{main1} proves the instability of certain singular critical points of the potential in the context of Newtonian dynamics, see corollary \ref{Generalization2} below. Theorem \ref{main1} also generalizes the immediate codimension $k>1$ version of \cite{BMP}, now in the Lagrangian dynamics of the mechanical Lagrangian \eqref{Mech_Lagrangian}. Section \ref{Corollary_proof} is devoted to the proof of the following result:

\begin{corollary}\label{Generalization}
Consider a potential of the form $U=g\circ F$ where $g:\R^{k}\rightarrow\R$ is a $C^{2}$ nonnegative function vanishing only at the origin, $F=(F_1,\ldots F_k):N\rightarrow\R^{k}$ is a $C^{3}$ function with $k<\dim N$ such that ${\bf 0}$ is a regular value, the vector fields $\nabla_{\rho}\, F_i\,/\,\Vert\,\nabla_{\rho}\, F_i\,\Vert^2$ pairwise commute and 
$$\left\langle\, \nabla_{\rho}\, F_i\, , \, \nabla_{\rho}\, F_j\,\right\rangle_\rho\equiv 0,\qquad i\neq j.$$
Then, every point of $M\times \{ {\bf 0} \}$ is a Lyapunov unstable equilibrium point of the Lagrangian dynamics of \eqref{Mech_Lagrangian}.
\end{corollary}

In the previous corollary, $M$ is a codimension $k$ submanifold of $N$ with dimension at least one.

Another interesting application of Theorem \ref{main1} is provided by a quasi-homogeneous potential in the context of Newtonian dynamics in $\R^{n}$:
$$U(\lambda^{\alpha_1}x_1,\ldots, \lambda^{\alpha_n}x_n)= \lambda^{r}\,U(x_1,\ldots, x_n),\quad \lambda\geq 0.$$
By the Euler's Theorem, the function
$$f({\bf x})= \alpha_1\frac{x_1^{2}}{2}+\ldots+\alpha_n\frac{x_n^{2}}{2}$$
verifies the hypothesis of Theorem \ref{main1} on every point where it is regular for
$$dU(\nabla f)= \left\langle \nabla U, \nabla f\right\rangle= \sum_{i=1}^{n} \alpha_i\, x_i\, \partial_i U= r\, U.$$
Denote by $V_{\boldsymbol{\alpha}}$ the set where the function $f$ is regular, i.e. the set of points ${\bf x}$ where $\alpha_i x_i\neq 0$ for some index $i$. Here $\boldsymbol{\alpha}$ denotes the vector $(\alpha_1,\ldots,\alpha_n)$. We have proved:

\begin{corollary}\label{Generalization2}
Consider a quasi-homogeneous potential $U$ with vector $\boldsymbol{\alpha}$. Then, every point in $(U^{-1}(0)\cap V_{\boldsymbol{\alpha}})\times \{{\bf 0}\}$ is a Lyapunov unstable equilibrium point of the Newtonian dynamics.
\end{corollary}

The interesting fact about the previous corollary is that it includes singular points of $U^{-1}(0)$. As an example consider the potential
$$U(x,y,z)= (x^{2}-y^{2}z)^{2}$$
where the minimum is attained at the \textit{Whitney umbrella}, the real algebraic variety $x^{2}-y^{2}z=0$, where every point in the $z$ axis with $z\geq 0$ is singular. Here, the nonzero vector $\boldsymbol{\alpha}$ lies in the linear span of $(1,1,0)$ and $(1,0,2)$. By the Corollary, every point of this variety distinct from the origin is unstable.

As another interesting example, consider the potential
$$U(x,y,z)=(x^{2}z^{2}+x^{3}-y^{2})^{2}$$ where the minimum is attained at the \textit{Kolibri}, the real algebraic surface $y^{2}= x^{2}z^{2}+x^{3}$. Here $\boldsymbol{\alpha}$ is nonzero and lies in the span of $(2,3,1)$. Again, every point of this variety distinct from the origin is unstable.


The situation is completely different for charged particles under the presence of magnetism mainly due to its general property of stabilization. In effect, even with a potential verifying the instability hypothesis posed in \cite{BMP}, it may happen that every point in $M\times\{\bf 0\}$ were Lyapunov stable. As an example of this phenomenon, consider the following potential energy and magnetic potential in the context of Newtonian dynamics in $\R^{3}$:
\begin{equation}\label{example}
U=z^{2},\qquad \mu= x\,dy.
\end{equation}
In this case, $M$ is the $z=0$ plane and every orbit of a point in this plane with initial velocity $v$ projects into the plane as a circle whose radius is the norm of $v_{\Vert}$ and projects into the $z$ axis as a linear pendulum motion with amplitude the norm of $v_{\perp}$. Here, $v=v_{\Vert}+v_{\perp}$ is the orthogonal splitting where the first term is in the $z=0$ plane and the second is in the $z$ axis. In particular, the equilibrium points in the plane $z=0$ are Lyapunov stable. This simple example also shows that Routh's converse is false in general.

Because the result in \cite{BMP} is a particular case of Corollary \ref{Generalization} we conclude that in presence of magnetism this corollary does not hold without further hypotheses on the magnetic field. Since the potential in \eqref{example} is homogeneous, this also shows that Corollary \ref{Generalization2} does not hold either without further hypotheses. In effect, taking the vector $\boldsymbol{\alpha}=(1,1,1)$, Corollary \ref{Generalization2} asserts that every point $(p,{\bf 0})$ with $p$ in the $z=0$ plane and distinct form the origin is unstable which is clearly false under the presence of magnetism in view of example \eqref{example}.

We extend Theorem \ref{main1} for charged particles under the presence of an additional magnetostatic field and prove the following

\begin{main}\label{main}
Consider a zero potential point $p$ and suppose there is a neighborhood $V$ of $p$ in $N$ and a real valued $C^{3}$ regular function $f$ on $V$ such that
$$dU(\,\nabla_\rho f\,)=O(U),\qquad \Vert\,\iota_{\nabla_\rho f}\,d\mu\,\Vert_{TN^{*}}= O(U^{1/2}).$$
Then, $(p,{\bf 0})$ is a Lyapunov unstable equilibrium point of the Lagrangian dynamics of \eqref{Lagrangian}.
\end{main}


Let us give now a direct application of the previous theorem. As in the previous example, consider the potential $U=z^2$ in $\R^3$ in the context of Newtonian dynamics but with the magnetic potential $\mu=z\,dy$ instead. Now the magnetic field is given by $d\mu=dz\,dy$. Again, every point in the $z=0$ plane is an equilibrium point but now, in contrast with the previous example, they are all Lyapunov unstable. In effect, consider the function $f=x$ whose gradient equals the first canonical vector $e_1$ and note that
$$dU(\,\nabla f\,)=0,\qquad \iota_{\nabla f}\,d\mu= {\bf 0}.$$

The following corollary constitutes the corresponding extension of Corollary \ref{Generalization} and will be proved in section \ref{Corollary3_proof}. It gives a family of examples of Theorem \ref{main}.

\begin{corollary}\label{Generalization3}
Under the hypotheses in Corollary \ref{Generalization} with a magnetic potential defined as the pullback of any one-form in $\R^k$ by $F$, that is to say
\begin{equation}\label{form_example}
\mu=F^*(\omega),\qquad \omega\in \Omega^1(\R^k),
\end{equation}
every point of $M\times \{ {\bf 0} \}$ is a Lyapunov unstable equilibrium point of the Lagrangian dynamics of \eqref{Lagrangian}. Moreover, if $d\omega$ is non null at ${\bf 0}$, then $d\mu$ is non null at every point in $M$. In particular for $k$ greater than one, there are magnetic potentials of the form \eqref{form_example} such that the corresponding magnetic field is non null at every point in $M$.
\end{corollary}

In \cite{Hagedorn} and its generalizations \cite{BK}, \cite{Fu}, \cite{Ref6}, \cite{So}, the condition required on the magnetic term for the instability to hold at a point necessarily implies the vanishing of the magnetic field at the point. The closest statement to that condition but in our formalism would be
\begin{equation}\label{condition}
\Vert\, d\mu\,\Vert_{\Lambda^{2}\,TN^{*}}= O(U^{1/2}).
\end{equation}

However, it is clear that the second equation in the hypothesis of Theorem \ref{main} generalizes the previous condition. Moreover, the magnetic field is not required to vanish at the point in question as in the previous condition. On the contrary, the magnetic field could be distinct from zero verifying the characteristic equation on the spatial projection $M$ of the set of equilibrium points $M\times \{\bf 0\}$:
\begin{equation}\label{characteristic}
(\iota_{\nabla_\rho f}\,d\mu)_x=0,\qquad x\in M.
\end{equation}

This case, the prevalence of the magnetic field over the potential one, was also treated in \cite{BN}, \cite{BN2} and \cite{So2}. In these references, the conditions required for instability are non degeneracy conditions on the potential and the magnetic field as well. However, in contrast to these conditions, condition \eqref{characteristic} extends non trivially to the purely magnetic Lagrangian where $U\equiv 0$ and $M$ is the whole manifold $N$. Degenerate conditions were treated by Kozlov in \cite{Kozlov_dissipation} under the presence of dissipative forces, something we do not assume in this paper.


Concerning the existence of a solution of the linear PDE \eqref{PDE}, Lewy's example shows that it is not always possible in the case where the coefficients are smooth \cite{Lewy}, that is to say, the case where the potential and the metric are smooth which is our case. In the case where the potential and the metric are real analytic, one is tempted to think that the Cauchy-Kovalevskaya existence Theorem would provide a real analytic solution of \eqref{PDE}. However, the boundary conditions on the coefficients at the locus of zero potential points are singular hence the hypotheses of the theorem do not hold \cite{Cauchy_K}. Even if there were a solution of \eqref{PDE}, Theorem \ref{main1} requires it to be regular at the point $p$.

Finally, there are unstable equilibrium points where there is no neighborhood verifying the hypothesis of Theorem \ref{main1}. As an example consider the potential $U(x,y)=x^{2}\, y^{2}$ in Newtonian dynamics. The axes $x=0$ and $y=0$ are invariant zero potential spaces hence the origin is unstable for these axes provide escape routes to infinity with arbitrary small velocity. However, there is no regular function at the origin verifying the hypothesis of Theorem \ref{main1} for every vector field tangent to both axes must be null at the origin. In particular, Theorem \ref{main1} provides a sufficient but not a necessary conditions for Lyapunov instability.

\section{Proof of the Main Theorem}

Let $p$ be a point in $M$ and for every $\varepsilon>0$ consider the solution $x^{\varepsilon}$ of the Euler-Lagrange equations of the Lagrangian \eqref{Lagrangian} with initial conditions $x^{\varepsilon}(0)=p$ and $\dot{x}^{\varepsilon}(0)=\varepsilon\, \nabla_\rho f(p)$.

For every $\varepsilon>0$, define $x_\varepsilon$ such that $x_\varepsilon(\tau)= x^{\varepsilon}(\tau/\varepsilon)$ where $x^{\varepsilon}$ is defined. These are solutions of the Euler-Lagrange equations of the Lagrangian:
\begin{equation}\label{Lagrangian2}
L_\varepsilon({\rm x}, {\rm v})= Q_{{\rm x}}({\rm v})+\varepsilon^{-1}\,\mu_{{\rm x}}({\rm v})-\varepsilon^{-2}\, U({\rm x}),\qquad ({\rm x}, {\rm v})\in TN.
\end{equation}
Now, the initial conditions $x_{\varepsilon}(0)=p$ and $\dot{x}_{\varepsilon}(0)= \nabla_\rho f(p)$ are fixed but the motion equations become singular as $\varepsilon\to 0^{+}$. Denote by $I_\varepsilon$ the maximal interval containing zero where $x_{\varepsilon}$ is defined.

\begin{lemma}\label{Lema1}
For every $\varepsilon>0$, $\Vert\dot{x}_\varepsilon(\tau)\Vert\leq \Vert \nabla_\rho f(p) \Vert$ for every $\tau$ in $I_\varepsilon$ and
$${\rm Im}(x_\varepsilon)\subset [U\leq \varepsilon^{2}\,\Vert \nabla_\rho f(p) \Vert^{2}/2].$$
\end{lemma}
\begin{proof}
For every $\varepsilon>0$, the Hamiltonian
$$H_\varepsilon({\rm x},{\rm v})= \Vert {\rm v} \Vert^{2}/2+\varepsilon^{-2}U({\rm x})$$
is constant along the solution $x_\varepsilon$ hence
$$\Vert \dot{x}_\varepsilon(\tau) \Vert^{2}/2,\ \varepsilon^{-2}U(x_\varepsilon(\tau))\leq H_\varepsilon(x_\varepsilon(\tau), \dot{x}_\varepsilon(\tau))= H_\varepsilon(p, \nabla_\rho f(p))= \Vert \nabla_\rho f(p) \Vert^{2}/2$$
and the result follows.
\end{proof}

\begin{corollary}\label{Cor1}
Let $T>0$. For every $\varepsilon>0$ and every $\tau$ in $I_\varepsilon\cap [-T, T]$, $(x_{\varepsilon}(\tau), \dot{x}_\varepsilon(\tau))\in R_T$ where
$$R_T= \left\lbrace\,({\rm x}, {\rm v})\in TN\ |\ {\rm x}\in \overline{B_\rho(p, T\, \Vert \nabla_\rho f(p) \Vert)},\ {\rm v}\in  \overline{B_{{\rm x}}({\bf 0}, \Vert \nabla_\rho f(p) \Vert)}\,\right\rbrace.$$
Note that this region is a compact set not depending on $\varepsilon$.
\end{corollary}
\begin{proof}
By Lemma \ref{Lema1}, $\Vert \dot{x}_\varepsilon(\tau)\Vert\leq \Vert \nabla_\rho f(p) \Vert$ and
$$d( x_\varepsilon(\tau),\,p) \leq \left\vert\int_0^{\tau}ds\ \Vert\dot{x}_\varepsilon(s)\Vert\ \right\vert\leq \vert \tau\vert\, \Vert \nabla_\rho f(p) \Vert\leq T\, \Vert \nabla_\rho f(p) \Vert,$$
the result follows.
\end{proof}

\begin{corollary}\label{Cor2}
For every $\varepsilon>0$, $x_\varepsilon$ is defined over the whole real line.
\end{corollary}
\begin{proof}
Consider the maximal interval $I_\varepsilon=(\omega_-, \omega_+)$ and suppose that $\omega_+$ is finite. Then, $(x_\varepsilon, \dot{x}_\varepsilon)|_{[0,\omega_+)}$ is contained in the compact set $R_{\omega_+}$ which is absurd hence $\omega_+=+\infty$. Analogously, $\omega_-=-\infty$.
\end{proof}

\begin{corollary}\label{Cor3}
Let $T>0$. There is a continuous curve $x:[-T,T]\rightarrow M$ with $x(0)=p$ and a sequence $(\varepsilon_j)$ such that $\varepsilon_j>0$, $\varepsilon_j\to 0^{+}$ and $x_{\varepsilon_j}\rightarrow x$ uniformly on $[-T,T]$.
\end{corollary}
\begin{proof}
Consider the family $\mathcal F$ of functions $x_{\varepsilon}$ defined in $[-T, T]$  that are solutions of the Euler-Lagrange equations of  \eqref{Lagrangian2} such that 
$x_{\varepsilon}(0)=p$ and $\dot x_{\varepsilon}(0)= \nabla_\rho f(p)$.  By Lemma \ref{Lema1} the family $\mathcal F$ is equicontinuous.  
Consider a sequence    $(x_{\varepsilon_j})$ such that $\varepsilon_j\to 0^{+}$.       By  Arzel\`a--Ascoli Theorem there is a subsequence, that we still call by   $(x_{\varepsilon_j})$, 
converging uniformly to a continuous curve $x$.  Because $x_\varepsilon(0)=p$ for every $\varepsilon>0$ we conclude that $x(0)=p$. 
Let  $t\in [-T, T]$. By Lemma \ref{Lema1} 

$$0\leq U(x_{\varepsilon_j}(t))\leq \varepsilon_j^{2}\,\Vert \nabla_\rho f(p) \Vert^{2}/2.$$

Therefore $U(x(t))=0$ for every $t\in [-T, T]$, i.e.,  

$${\rm Im}(x)\subset   [U=0]=M$$
and this finishes the proof.
\end{proof}

Now we construct suitable coordinates. Recall that $V$ is a neighborhood of $p$ in $N$ where the regular $C^{3}$ function $f$ is defined. Consider the flow $\phi$ in $V$
\begin{equation}\label{Cauchy3}
\partial_t \phi= \frac{\nabla_\rho f}{\Vert \nabla_\rho f \Vert^{2}}(\phi),\ \phi(0,{\rm x})= {\rm x},\ {\rm x}\in V.
\end{equation}

Without loss of generality, by adding a constant if necessary we may suppose that $f(p)=0$. Consider a local parameterization $(W, \psi)$ of the submanifold $f^{-1}(0)$ in $V$ centered at $p$. Let $B$ be the open set of those $(z, y)$ such that $(z,\psi(y))$ belongs to the domain of $\phi$. 
Define $\Psi:B\rightarrow N$ by 
\begin{equation}\label{Def_Psi}
\Psi(z,y)= \phi(z,\psi(y))
\end{equation}

\begin{lemma}\label{Lema2}
\begin{enumerate}
\item $\Psi(0,y)= \psi(y)$ for every $y$ in $W$.
\item $f(\Psi(z,y))= z$ for every $(z,y)$ in $B$.
\item $(B, \Psi)$ is a $C^{2}$ local parameterization of $N$ centered at $p$.
\end{enumerate}
\end{lemma}
\begin{proof}
\begin{enumerate}
\item This is immediate from the definition \eqref{Cauchy3}.
\item By definition, $\partial_t(f\circ\phi)= \langle\nabla_\rho f(\phi), \partial_t\phi\rangle= 1$ hence
$$f\circ\phi(t, x)= t+ f(\phi(0,x))= t+f(x),$$
$$f(\Psi(z,y))= z+f(\psi(y))=z$$
for $\psi(y)$ is in $f^{-1}(0)$.

\item Because $\nabla_\rho f(\psi(y))\neq {\bf 0}$ and $\nabla_\rho f(\psi(y))\perp T_{\psi(y)}\psi(W)$, we conclude that $d_{(0,y)}\Psi$ is an isomorphism for $d_{y}\psi$ is so.

Consider a point $(z,y)$ in $B$. Let $s\mapsto (\alpha(s),\beta(s))\in B$ be a curve with initial conditions $\alpha(0)=0$ and $\beta(0)=y$ such that $(z+\alpha(s),\beta(s))\in B$. 
By \eqref{Def_Psi} we have

$$\Psi(z+\alpha(s), \beta(s))=\phi_z(\Psi(\alpha(s), \beta(s))). $$ 

where we have defined $\phi_z(a)= \phi(z,a)$. Hence 
$$d_{(z,y)}\Psi(\dot\alpha(0),\dot\beta(0))= d_{\psi(y)}\phi_z( d_{(0,y)}\Psi (\dot\alpha(0),\dot\beta(0))), $$
 i.e.

$$d_{(z,y)}\Psi = d_{\psi(y)}\phi_z\circ d_{(0,y)}\Psi.$$
Then, $d_{(z,y)}\Psi$ is also an isomorphism for every $d_{\psi(y)}\phi_z$ is so by the Liouville formula. By the inverse function Theorem, $\Psi$ is a local diffeomorphism. To show that $\Psi$ is an embedding, it rest to show that it is injective.

Suppose that $\Psi(z,y)=\Psi(z', y')$. Then,
$$z= f(\Psi(z,y))= f(\Psi(z',y'))= z',$$
$$\psi(y)= \phi_{-z}(\Psi(z,y))= \phi_{-z}(\Psi(z,y'))= \psi(y')$$
so $(z,y)=(z', y')$ for $\psi$ is injective.
\end{enumerate}
\end{proof}

Let $T>0$ be small enough such that the compact set $\overline{B_{\rho}(p, T\,\Vert \nabla_\rho f(p) \Vert)}$ is contained in $\Psi(B)$. From now on, all the curves will be defined on $[-T,T]$.

Denote by $z_\varepsilon$ and $y_\varepsilon$ the coordinates of $x_\varepsilon$ with respect to the parameterization $(B, \Psi)$:
$$\Psi(z_\varepsilon(\tau), y_\varepsilon(\tau))= x_\varepsilon(\tau),\qquad \tau\in [-T,T]$$
and an analogous definition for the limit curve in Corollary \ref{Cor3}.

In what follows, we adopt the Einstein's summation convention on repeated indices. 

For every $\xi=(z,y)$ in $B$ we define the zero superindex of $\xi$ as the $z$ coordinate: $\xi^{0}=z$. This way, we denote by greek letters indices ranging from $0$ to $n-1$. We denote by latin letters indices ranging from $1$ to $n-1$, thus $\xi^b=y^b$ for $i=1,\dots,n-1$.

Because of the construction of the parameterization, the pullback of the metric $\rho$ by the map $\Psi$ has the form
$$\Psi^{*}(\rho)= \gi_{00}\, dz\otimes dz + \gi_{ab}\,dy^{a}\otimes dy^{b}.$$
Note that there are no mixed indices terms or equivalently $\gi_{0a}=0$. In particular, thinking of the metric as a matrix $(\gi_{\alpha\beta})$, its inverse $(\gi^{\alpha\beta})$ verifies
$$\gi^{00}= \gi_{00}^{-1},\quad (\gi^{ab})=(\gi_{ab})^{-1},\quad \gi^{0a}=0.$$

The Christoffel symbols (of the second kind) are defined as follows:
$$\Gamma^{\alpha}_{\mu\nu}=\frac{\gi^{\alpha\beta}}{2}\left(\gi_{\beta\mu,\nu}+\gi_{\beta\nu,\mu}- \gi_{\mu\nu,\beta}\right)$$
where we have denoted by a comma the respective partial derivative. In contrast with the metric, now these symbols are not the coefficients of a tensor. However, with respect to the parameterization we are using, they are the coefficients of a connection, specifically, the Levi-Civita connection.

The pullback of the magnetic form $\mu$ by the parameterization considered before has the form
$$\Psi^{*}(\mu)= \mathrm{A}_\alpha\, d\xi^{\alpha}$$
and the magnetic field
$$\Psi^{*}(d\mu)= \frac{1}{2}\Fi_{\alpha\beta}\ d\xi^{\alpha}\wedge d\xi^{\beta}= \sum_{\alpha<\beta}\Fi_{\alpha\beta}\ d\xi^{\alpha}\wedge d\xi^{\beta}$$
where Einstein's convention were not used in the last equality. In particular,
$$\Fi_{\alpha\beta}= \partial_\alpha \mathrm{A}_\beta - \partial_\beta \mathrm{A}_\alpha.$$

\begin{lemma}\label{Lema3}
\begin{enumerate}
\item The functions $\dot{z}_\varepsilon$ and $\dot{y}_\varepsilon^{k}$ are uniformly bounded on $[-T,T]$ by a constant not depending on $\varepsilon>0$.
\item There is a constant $C_1>0$ not depending on $\varepsilon>0$ such that $\vert\,\Fi_{0a}(z_\varepsilon, y_\varepsilon)\,\vert\leq C_1\,\varepsilon$ on $[-T,T]$ for every $\varepsilon>0$.
\item Define $u=U\circ \Psi$. There is a constant $C_2>0$ not depending on $\varepsilon>0$ such that $\vert\, \partial_0 u(z_\varepsilon, y_\varepsilon)\,\vert\leq C_2\,\varepsilon^{2}$ on $[-T,T]$ for every $\varepsilon>0$.
\end{enumerate}
\end{lemma}
\begin{proof}
For every ${\rm x}$ in $B$ define the quadratic form $Q_{{\rm x}}$ such that
$$Q_{{\rm x}}({\rm v})= \Vert d_{{\rm x}} \Psi({\rm v})\Vert^{2}.$$
It is positive definite for every ${\rm x}$ and defines a strictly positive continuous function on the unit tangent sphere bundle $\pi: T^{1}B\to B$ with respect to the euclidean metric on $T\,B$. In particular, it attains a minimum and maximum value ${\rm m,\ M}>0$ respectively on the compact set $(\Psi\circ\pi)^{-1}(\overline{B_\rho(p,T\,\Vert \nabla_\rho f(p) \Vert )})$. Denote by $m_\nabla>0$ the minimum of $\Vert \nabla_\rho f \Vert$ on $\overline{B_\rho(p,T\,\Vert \nabla_\rho f(p) \Vert )}$.
\begin{enumerate}
\item For every $\varepsilon>0$ and $\tau$ in $[-T,T]$ we have
$${\rm m}\,\Vert(\dot{z}_{\varepsilon}(\tau), \dot{y}_{\varepsilon}(\tau))\Vert_E^{2}\leq Q_{(z_{\varepsilon}(\tau), y_{\varepsilon}(\tau))}(\dot{z}_{\varepsilon}(\tau), \dot{y}_{\varepsilon}(\tau))= \Vert \dot{x}_\varepsilon(\tau)\Vert^{2}\leq \Vert \nabla_\rho f(p) \Vert^{2}$$
where $\Vert\cdot \Vert_E$ is the euclidean norm on $T\,B$ hence
$$\vert\dot{z}_{\varepsilon}(\tau)\vert,\ \vert\dot{y}_{\varepsilon}^{k}(\tau))\vert\leq \frac{\Vert \nabla_\rho f(p) \Vert}{{\rm m}^{1/2}}.$$

\item

By our hypothesis in Theorem \ref{main}, there is a constant $K_1>0$ not depending on $\varepsilon$ such that 

$$ \Vert \iota_{\nabla_\rho f} d\mu({\rm x})\Vert^{2}_{T^{*}\Psi(B)}\leq K_1\, U({\rm x})$$
for every ${\rm x}$ in $\Psi(B)$ where the norm on the cotangent space $T^{*}\Psi(B)$ is the canonical one induced from the norm on the tangent space through the Riesz representation Theorem. For every $\varepsilon>0$ and $\tau$ in $[-T,T]$ we have
$${\rm M}^{-1}\,\sum_a (\Fi_{0a})^{2}\leq \gi^{ab}\,\Fi_{0a}\,\Fi_{0b} =
\Vert \Fi_{0a}\, dy^{a}\Vert^{2}_{T^{*}B}= \Vert \iota_0 \Psi^{*}d\mu\Vert^{2}_{T^{*}B}$$
$$= \Vert \nabla_\rho f \Vert^{-4}\,\Vert \iota_{\nabla_\rho f} d\mu\Vert^{2}_{T^{*}\Psi(B)}
\leq m_\nabla^{-4}\,\Vert \iota_{\nabla_\rho f} d\mu\Vert^{2}$$
$$\leq m_\nabla^{-4}K_1\, U(x_\varepsilon)\leq m_\nabla^{-4}K_1\frac{\Vert\nabla_\rho f(p)\Vert^2}{2}\, \varepsilon^{2}.$$
where the norm $\Vert \cdot\Vert_{T^{*}B}$ on the cotangent space $T^{*}B$ is the pushout of the norm $\Vert \cdot\Vert_{T^{*}\Psi(B)}$ by the codiferential $\delta\Psi$. Recall that this pushout is possible for $\Psi$ is a diffeomorphism. Set 
$$ C_1=  \left(\frac{ m_\nabla^{-4}K_1}{2{\rm M}}\right)^{1/2}\Vert\nabla_\rho f(p)\Vert,$$
hence
$$\vert\Fi_{0a}\vert\leq C_1\, \varepsilon.$$

\item

By our hypothesis in Theorem \ref{main}, there is a constant $K_2>0$ not depending on $\varepsilon$ such that 
$$\vert dU( \nabla_\rho f({\rm x}))\vert\leq K_2\, U({\rm x})$$
for every ${\rm x}$ in $\Psi(B)$.  Set
$$C_2=m_\nabla^{-2}  K_2\frac{\Vert\nabla_\rho f(p)\Vert^2}{2}.$$
Hence
$$\vert\partial_0 u\vert= \vert dU(\partial_0 \Psi)\vert= \Vert \nabla_\rho f \Vert^{-2}\, \vert dU(\nabla_\rho f)\vert \leq m_\nabla^{-2}\,K_2\,U(x_\varepsilon)\leq \,C_2\,\varepsilon^{2}.$$

\end{enumerate}
The proof is complete.
\end{proof}


\begin{lemma}\label{Lema5}
The functions $z_{\varepsilon}$ and $y_{\varepsilon}^{k}$ satisfy the equation

\begin{equation}\label{Equation_y}
\ddot{z}_{\varepsilon}+  \Gamma^{0}_{0 0}\,\dot{z}^2_{\varepsilon}+ 2\,\Gamma^{0}_{0 b}\,\dot{z}_{\varepsilon}\,\dot{y}^{b}_{\varepsilon} +\Gamma^{0}_{a b}\,\dot{y}^{a}_{\varepsilon}\,\dot{y}^{b}_{\varepsilon}+
\varepsilon^{-1}\, \gi^{00}\,\Fi_{0 b}\,\dot{y}^{b}_{\varepsilon}+ \varepsilon^{-2}\,\gi^{00}\,\partial_{0} u=0\;\;\;\;\;\;\;\;\;\;
\end{equation}
where the coefficients are evaluated over the curve $(z_\varepsilon, y_\varepsilon)$.
\end{lemma}
\begin{proof}
With respect to the parameterization $(B, \Psi)$, the Lagrangian \eqref{Lagrangian2} reads as follows:
$$L_\varepsilon\circ d\Psi(z,y,v)= \frac{\gi_{\alpha\beta}(z,y)}{2}\,v^{\alpha}v^{\beta}+\varepsilon^{-1}\,A_\alpha(z,y)\, v^{\alpha} -\varepsilon^{-2}\,u(z,y).$$
The Euler-Lagrange equations are covariant with respect to coordinate changes hence the coordinate curves $z_{\varepsilon}$ and $y_{\varepsilon}$ verify them. A straightforward calculation of the equation gives the result.
\end{proof}

\begin{corollary}\label{Cor4}
There is a positive constant $C$ independent of $\varepsilon>0$ such that $\vert\ddot{z}_{\varepsilon}(\tau)\vert\leq C$ for every $\tau$ in $[-T,T]$.
\end{corollary}
\begin{proof}
All of the coefficients $\Gamma^0_{00}$,  $\Gamma^{0}_{0 b}$, $\Gamma^{0}_{a b}$ and  $g^{00}$ are continuous on $B$ hence they are bounded on the compact set $\Psi^{-1}(\overline{B_\rho(p,T\,\Vert \nabla_\rho f(p) \Vert )})$. Therefore they are bounded over the curve $(z_{\varepsilon}, y_{\varepsilon})$. By Lemma \ref{Lema3} (ii) (iii), the coefficients 
$\varepsilon^{-2}\,\partial_{0} u$ and $\varepsilon^{-1}\,\Fi_{0 b}$ are bounded on $B$ by a constant not depending on $\varepsilon>0$. 
The velocities $\dot z_{\varepsilon}$ and $\dot y^k_{\varepsilon}$ are bounded by a constant independent of $\varepsilon$  on account of 
Lemma \ref{Lema3} (i).  
Therefore, by Lemma \ref{Lema5} the same occurs with the accelerations and we have the result.
\end{proof}

\begin{proposition}\label{Clave}
There is a continuous curve $x:[-T,T]\rightarrow M$ with $x(0)=p$ and a sequence $(\varepsilon_j)$ such that $\varepsilon_j>0$, $\varepsilon_j\to 0^{+}$, $f\circ x_{\varepsilon_j}\rightarrow f\circ x$ in $C^{1}[-T,T]$ and the derivative of $f\circ x$ at zero is strictly positive.
\end{proposition}
\begin{proof}
By Lemma \ref{Lema2}(ii) we have

$$f(x_{\varepsilon}(t))=f(\Psi(z_{\varepsilon}(t),y_{\varepsilon}(t))=z_{\varepsilon}(t).$$

Since $\dot x_{\varepsilon}(0)= \nabla_\rho f(p)$ for every $\varepsilon>0$ we conclude that $\dot{z}_\varepsilon(0)=\Vert \nabla_\rho f(p) \Vert^{2}$
for every $\varepsilon>0$. 
Let $\varepsilon_j$ be the sequence obtained in  Corollary \ref{Cor3}. 

By  Corollary \ref{Cor4} and Arzel\`a--Ascoli Theorem, taking a subsequence if necessary, we have $\dot{z}_{\varepsilon_j}\rightarrow e$ uniformly on $[-T,T]$ for some continuous function $e$. Because $\dot{z}_\varepsilon(0)=\Vert \nabla_\rho f(p) \Vert^{2}$ is independent of $\varepsilon>0$, we have $e(0)=\Vert \nabla_\rho f(p) \Vert^{2}$. For every $j$ we have
$$z_{\varepsilon_j}(\tau)= \int_0^{\tau} ds\ \dot{z}_{\varepsilon_j}(s)$$
and taking the limit as $j\to +\infty$,
$$z(\tau)= \int_0^{\tau} ds\ e(s).$$
In particular, $\dot{z}=e$ is continuous and $\dot{z}(0)=\Vert \nabla_\rho f(p) \Vert^{2}>0$. This concludes the proof.
\end{proof}

\begin{proof}[Proof of the Main Theorem]
By Proposition \ref{Clave}, $f\circ x|_{[0,T]}$ attains a maximum $M>0$ at $\tau_*>0$ and there is $j_0$ such that $f(x_{\varepsilon_j}(\tau_*))>M/2$ if $j\geq j_0$. Therefore, $f^{-1}(-M/2,M/2)$ is a neighborhood of $p$ and
$$x^{\varepsilon_j}(\tau_*/\varepsilon_j)\notin f^{-1}(-M/2,M/2),\qquad j\geq j_0$$
while $(x^{\varepsilon_j}(0), \dot{x}^{\varepsilon_j}(0))= (p, \varepsilon_j \nabla_\rho f(p))\rightarrow (p,{\bf 0})$ as $j\to +\infty$. We conclude that $(p,{\bf 0})$ is a Lyapunov unstable equilibrium point. The proof is complete.
\end{proof}

\section{Proof of Corollary \ref{Generalization}}\label{Corollary_proof}

Along this section, we will simply denote as $\nabla$ the gradient with respect to the metric $\rho$ and all the norms and inner product will refer to this metric as well.

Consider a point $p$ in $M$. We will show that under the hypothesis of the corollary, the hypothesis of Theorem \ref{main1} hold on $p$ hence $(p,{\bf 0})$ is Lyapunov unstable.

The argument for the construction of a local parameterization around the zero potential point $p$ is almost verbatim to the one before in Lemma \ref{Lema2} except for the fact that now, every $F_i$ generates a flow $\phi^{i}$ as follows
\begin{equation}\label{Cauchy4}
\partial_t \phi^{i}= \frac{\nabla F_i}{\Vert \nabla F_i \Vert^{2}}(\phi^{i}),\ \phi^{i}(0,x)= x,\ x\in \R^{n}-Crit(F_i).
\end{equation}
In particular, because of the orthogonality of the gradients $\nabla F_i$, we have the identity
$$\partial_t(F_j\circ\phi^{j})= \langle\,\nabla F_j(\phi^{i})\, , \, \partial_t\phi^{i}\,\rangle= \delta^{\,i}_j$$
and integrating this expression we immediately conclude that
\begin{equation}\label{gen_eq1}
F_j\circ\phi^{i}(t, x)= \delta_j^{i}\, t+ F_j(\phi^{i}(0,x))= \delta_j^{\,i}\, t+ F_j(x)
\end{equation}
where $\delta_j^{\,i}$ is the Kronecker's delta on the indices.

Consider a local parameterization $(V, \psi)$ of $M$ centered at $p$ with $V\subset \R^{n-k}$ and define $(W\times V\, , \, \Psi)$ where $W$ is some small enough neighborhood of the origin in $\R^{k}$ and $\Psi$ is defined by the following expression
\begin{equation}\label{Def_Psi_2}
\Psi(\boldsymbol{r}, y)= \phi^{1}_{r_1}\circ\ldots\circ \phi^{k}_{r_k}(\psi(y)).
\end{equation}
Here $\boldsymbol{r}$ denotes the vector $(r_1,\ldots, r_k)$ and we have defined $\phi^{i}_t(x)=\phi^{i}(t,x)$. An almost verbatim argument as in Lemma \ref{Lema2} shows that $(W\times V\, , \, \Psi)$ is a local parameterization and it is an extension of $(V, \psi)$ for $\Psi({\bf 0}, y)=\psi(y)$ for every $y$ in $V$.

Because of the commutativity hypothesis on the vector fields, the respective flows commute and they can be arranged whatsoever in the expression \eqref{Def_Psi_2}. In particular, we have the expression
\begin{equation}\label{partial_Psi_radial}
\partial_{r_i}\,\Psi(\boldsymbol{r}, y)= \frac{\nabla F_i}{\Vert \nabla F_i \Vert^{2}}\left(\Psi(\boldsymbol{r}, y)\right),\qquad i=1,\ldots, k.
\end{equation}

Because of the identities \eqref{gen_eq1}, which follow from the perpendicularity hypothesis on the vector fields, it follows that
\begin{equation}\label{Psi_enderza}
F(\Psi(\boldsymbol{r}, y))= \boldsymbol{r}.
\end{equation}
In particular, because of the identities \eqref{partial_Psi_radial}, \eqref{Psi_enderza} and the perpendicularity hypothesis on the vector fields, the pullback of the metric $\rho$ by the map $\Psi$ reads as follows
\begin{equation}\label{metric_pullback}
\Psi^{*}(\rho)= \gi_{11}\, dr_1\otimes dr_1 +\ldots + \gi_{kk}\, dr_k\otimes dr_k + \sum _{a,b=k+1}^{n}\,\gi_{ab}\,dy^{a-k}\otimes dy^{b-k}
\end{equation}
where Einstein's convention on repeated indices were not used in the last equality. Thinking of the metric as a matrix $(\gi_{ij})$, it has two blocks only: the upper left and the lower right. The first of these blocks is diagonal and the same structure appears in the inverse matrix $(\gi^{ij})$.

Now we define the function $f$ simply as one of the $y$ coordinates. Specifically, choose any coordinate $y^{j}$ and define
$$f= \pi_{k+j}\circ \Psi^{-1}:\Psi(W\times V)\rightarrow \R$$
where $\pi_{k+j}$ is the projection on the $(k+j)$-th coordinate, that is to say the $y^{j}$ coordinate. Denote the pullbacks of $U$ and $f$ by the map $\Psi$ as $\tilde{U}$ and $\tilde{f}$ respectively and note that $\tilde{U}= g\otimes {\bf 1}_{n-k}$ and does not depend on the $y$ coordinates while $\tilde{f}$ is just the $y^{j}$ coordinate where we have denoted by ${\bf 1}_{n-k}$ the unit constant on $\R^{n-k}$. Then,
$$\langle\,\nabla U, \nabla f\,\rangle\circ \Psi= \gi^{ab}\,\partial_a\tilde{U}\,\partial_b\tilde{f}= \gi^{aj}\,\partial_a\tilde{U}=0$$
for $\partial_a\tilde{U}=0$ when the index $a$ runs on $n-k+1,\ldots, n$ and $\gi^{aj}=0$ when it runs on $1,\ldots, k$ hence every term of the last equality is zero. Because $\Psi$ is a diffeomorphism, we conclude that
$$\langle\,\nabla U, \nabla f\,\rangle = 0$$
on the neighborhood $\Psi(W\times V)$ of the point $p$. The proof is complete.

\section{Proof of Corollary \ref{Generalization3}}\label{Corollary3_proof}

This is a continuation of the proof written in the previous section. Consider an arbitrary one\,-form in $\R^k$
$$\omega= \sum_{l=1}^k\,c_l\,dx^l$$
where $c_l$ are arbitrary smooth functions on $\R^k$. Then, because of the identity \eqref{Psi_enderza} we have the one\,-form in $W\times V$
$$\Psi^*(\mu)=\Psi^*F^*(\omega)=(F\circ\Psi)^*(\omega)=\sum_{l=1}^k\,c_l(r_1,\ldots, r_k)\,dr_l$$
whose exterior differential reads as
\begin{equation}\label{id0}
\Psi^*(d\mu)=d\,\Psi^*(\mu)=\sum_{i,l=1}^k\,\partial_i c_l(r_1,\ldots, r_k)\,dr_i\wedge dr_l.
\end{equation}

By definition of the inner contraction, we have the identity
\begin{equation}\label{id1}
\Psi^*(\iota_{\nabla f}\,d\mu)=\iota_{(\nabla f)_*}\,\Psi^*(d\mu),\qquad (\nabla f)_*= d\Psi^{-1}(\nabla f).
\end{equation}
We claim that the field $(\nabla f)_*$ coincides with the gradient of $\Psi^*\,f$ with respect to the metric $\Psi^{*}(\rho)$. In effect, for every vector $v$ in $T(W\times V)$ we have
$$(d\,\Psi^*f)(v)= (\Psi^*\,df)(v)= \langle\,\nabla f, d\Psi(v)\,\rangle_\rho= \langle\,d\Psi\left((\nabla f)_*\right), d\Psi(v)\,\rangle_\rho$$
$$= \langle\,(\nabla f)_*\,,\, v\,\rangle_{\Psi^*(\rho)}$$
and the non degeneracy of the metric proves the claim. In particular, because of the expression \eqref{metric_pullback} for the pullback of the metric and recalling that $f$ is just the $y_j$ coordinate with respect to $\Psi$, we conclude the vanishing of the radial components of the vector field $(\nabla f)_*$ for
$$(\nabla f)_*^a= \sum_b\,\gi^{ab}\,\partial_b \tilde{f}= \gi^{aj},$$
that is to say
\begin{equation}\label{id2}
(\nabla f)_*= \sum_{a=k+1}^n\,\gi^{aj}\,{\bf e}_a
\end{equation}
where we have denoted by ${\bf e}_a$ the canonical vector field corresponding to the $y^{a-k}$ tangential coordinate.

However, the inner contraction of any two\,-form $dr_i\wedge dr_l$ by a tangential canonical field ${\bf e}_a$ is identically zero for
\begin{equation}\label{id3}
\iota_{{\bf e}_a}\,\left( dr_i\wedge dr_l\right)= \left( dr_i\,({\bf e}_a)\right)\,dr_l - \left( dr_l\,({\bf e}_a)\right)\,dr_i={\bf 0},\quad a\geq k+1
\end{equation}
because the coefficients of both terms vanish.

By the identities \eqref{id0}, \eqref{id1}, \eqref{id2} and \eqref{id3} we have
$$\Psi^*(\iota_{\nabla f}\,d\mu)= \sum_{i,l=1}^k\,\sum_{a=k+1}^n\,\gi^{aj}\,\partial_i c_l(r_1,\ldots, r_k)\,\iota_{{\bf e}_a}\,\left(dr_i\wedge dr_l\right)= {\bf 0}$$
and we conclude that $\iota_{\nabla f}\,d\mu$ identically vanishes on $\Psi(W\times V)$ hence by Theorem \ref{main} the point $(p, {\bf 0})$ is Lyapunov unstable.

Finally, because the origin is a regular value of $F$, for every point $x$ in $M$ the set containing the differentials $d_x\,F_i$ is linearly independent hence the form
$$d\mu=d\,F^*(\omega)=F^*(d\omega)=\sum_{i,l=1}^k\,\partial_i c_l(F_1,\ldots, F_k)\,dF_i\wedge dF_l$$
is non null at every point in $M$ provided $d\omega$ is non null at the origin. This finishes the proof.


\medskip
Received xxxx 20xx; revised xxxx 20xx.
\medskip

\end{document}